\documentclass{amsart}
\usepackage{graphicx}
\usepackage{url}
\usepackage{amssymb}
\usepackage[dvipdfmx]{hyperref}
\usepackage{amssymb}
\usepackage{url}
\usepackage{cite}

\newtheorem{theorem}{Theorem}[section]
\theoremstyle{plain}

\newtheorem{corollary}{Corollary}[section]

\newtheorem{lemma}{Lemma}[section]

\newtheorem{problem}{Problem}

\numberwithin{equation}{section}
\begin{document}
\title[ On the super edge-magicness of graphs with a specific degree sequence]{ On the super edge-magicness of graphs with a specific degree sequence}
\author{ Rikio Ichishima}
\address{Department of Sport and Physical Education, Faculty of Physical
Education, Kokushikan University, 7-3-1 Nagayama, Tama-shi, Tokyo 206-8515,
Japan}
\email{ichishim@kokushikan.ac.jp}
\author{Francesc A. Muntaner-Batle}
\address{Graph Theory and Applications Research Group, School of Electrical
Engineering and Computer Science, Faculty of Engineering and Built
Environment, The University of Newcastle, NSW 2308 Australia }
\email{famb1es@yahoo.es}
\subjclass{05C07, 05C78}
\keywords{(super) edge-magic graph, (super) edge-magic labeling, vertex degree, degree sequence, graph labeling}

\begin{abstract}
A graph $G$ is said to be super edge-magic if there exists a bijective function $f:V\left(G\right) \cup E\left(G\right)\rightarrow \left\{1, 2, \ldots , \left\vert V\left( G\right) \right\vert +\left\vert E\left( G\right) \right\vert \right\}$ such that $f\left(V \left(G\right)\right) =\left\{1, 2, \ldots , \left\vert V\left( G\right) \right\vert \right\}$ and $f\left(u\right) + f\left(v\right) + f\left(uv\right)$ is a constant for each $uv\in E\left( G\right) $. 
In this paper, we study the super edge-magicness of graphs of order $n$ with degree sequence $s:4, 2, 2, \ldots, 2$. 
We also investigate the super edge-magic properties of certain families of graphs.
This leads us to propose some open problems.
\end{abstract}

\date{Aug 13, 2023}
\maketitle

\section{Introduction}

Only graphs without loops or multiple edges will be considered in this
paper. Undefined graph theoretical notation and terminology can be found in 
\cite{CL} or \cite{West}. 
The \emph{vertex set} of a graph $G$ is denoted by $V \left(G\right)$, 
while the \emph{edge set} of $G$ is denoted by $E\left (G\right)$.

Motivated by the notion of magic valuation of graphs introduced by Kotzig and Rosa \cite{KR}, 
Enomoto, Llad\'{o}, Nakamigawa, and Ringel \cite{ELNR} introduced the concept of super edge-magic graphs.
A graph $G$ is said to be \emph{super edge-magic} if there exists a bijective function 
$f:V\left(G\right) \cup E\left(G\right)\rightarrow \left\{1, 2, \ldots , \left\vert V\left( G\right) \right\vert +\left\vert E\left( G\right) \right\vert \right\}$ 
such that $f\left(V \left(G\right)\right) =\left\{1, 2, \ldots , \left\vert V\left( G\right) \right\vert \right\}$ and $f\left(u\right) + f\left(v\right) + f\left(uv\right)$ is a constant (called the \emph{valence} $\text{val}\left(f\right)$ of $f$) for each $uv\in E\left( G\right) $. 
Such a function is called a \emph{super edge-magic labeling}.

The following lemma found in \cite{FIM} is a very useful characterization of super edge-magic graphs.

\begin{lemma}
\label{trivial}
A graph $G$ is super edge-magic if and only if
there exists a bijective function $f:V\left( G\right) \rightarrow \left\{1, 2, \ldots , \left\vert V\left( G\right) \right\vert \right\} $ such that the set 
\begin{equation*}
S=\left\{ f\left( u\right) +f\left( v\right) : uv\in E\left( G\right) \right\}
\end{equation*}
consists of $\left\vert E\left( G\right) \right\vert $ consecutive integers. In such a case, $f$ extends to a super
edge-magic labeling of $G$ with valence $k=\left\vert V\left( G\right) \right\vert +\left\vert E\left( G\right) \right\vert +s$, where $s=\min
\left( S\right) $ and 
\begin{equation*}
S=\left\{ k-\left( \left\vert V\left( G\right) \right\vert +1\right), k-\left( \left\vert V\left( G\right) \right\vert +2\right), \ldots , k-\left( \left\vert V\left( G\right) \right\vert +\left\vert E\left( G\right) \right\vert \right) \right\} \text{.}
\end{equation*}
\end{lemma}

It is worth to mention that Acharya and Hegde \cite{AH} introduced the concept of strongly indexable graph that turns out to be equivalent to the concept of super edge-magic graph (see \cite{HS}). 
The study of super edge-magic labelings of graphs has proven to be crucial in the last two decades, since many relations with other types of labelings have been found (see \cite{ FIM}), and relations with other concepts such as Skolem and Langford sequences (see \cite{LM1}), and dual shuffle primes and Jacobsthal sequences (see \cite{LMP} and \cite{LMR3}). 

For a thorough study of graph labeling problems, see the extensive survey by Gallian \cite{Gallian}. 
For more information on super edge-magic graphs and related topics, see the books by Ba\v{c}a and Miller\cite{BM}, 
Chartrand, Egan, and Zhang \cite{CEZ}, L\'{o}pez and Muntaner-Batle\cite{LM2}, and Marr and Wallis \cite{MW}. 

One of the main goals of this paper is to study the super edge-magic properties of graphs of order $n$ with degree sequence $s:4, 2, 2, \ldots, 2$.
In particular, we would like to stress the family of graphs that we define next that will be a point of study in the subsequent pages of this paper. 
For integers $m \geq 3$ and $n \geq 3$, let $C(m,n)$ denote the graph consisting of two cycles $C_{m}$ and $C_{n}$ of orders $m$ and $n$ in such a way that the two cycles share exactly one vertex in common.
Observe that this family of graphs is the simplest case of the family of graphs for which all blocks are cycles. 
In fact, another goal for this paper is to spread around the interest about the super edge-magic properties of graphs of this type.

Next, we proceed with the following lemma found in \cite{FIM} that will be useful to obtain certain results in this paper.
For this purpose, the degree of a vertex $v$ is denoted by $\deg_{G} (v)$ or simply by $\deg (v)$ if the graph $G$ under discussion is clear.

\begin{lemma}
\label{not_SEM}
Let $G$ be a graph such that $\deg (v)$ is even for all $v \in V\left(G\right)$ and $ \left\vert E\left( G\right) \right\vert \equiv 2\pmod{4}$. 
Then $G$ is not super edge-magic.
\end{lemma}

\section{Main results}
We are now ready to state and prove our first result.

\begin{theorem}
\label{main_1}
Let $G$ be a graph of even order $n \geq 6$ with degree sequence $s:4, 2, 2, \ldots, 2$.
Then $G$ is not super edge-magic.
\end{theorem}
\begin{proof}
Consider such a graph $G$ with a super edge-magic labeling $f$.
Then
\begin{equation*}
\left\vert E\left( G\right) \right\vert = \sum_{v \in V\left(G\right)} \deg (v) = \frac{4+2\left(n-1\right)}{2}=n+1 \text{.}
\end{equation*} 
Thus, the valence $\text{val}\left(f\right)$ of $f$ can be computed as follows:
\begin{equation*}
\begin{split}
\text{val}\left(f\right)&=\frac{ 2\sum_{i=1}^{n} i +\sum_{i=n+1}^{2n+1} i +2\alpha}{n+1}\\
&=\frac{ \sum_{i=1}^{n} i +\sum_{i=1}^{2n+1} i +2\alpha}{n+1}=\frac{5n^2+7n+\left(2+4\alpha\right)}{2\left(n+1\right)} \text{,}
\end{split}
\end{equation*}
where $\alpha \in \left[1,n\right]$.
We next examine for which valences of $\alpha \in \left[1,n\right]$, $4\alpha$ is a multiple of $2\left(n+1\right)$ or, 
equivalently, $2\alpha$ is a multiple of $n+1$. 
Observe that if $\alpha \in \left[1,n/2\right]$, then we have $2\alpha \leq n <n+1$.
On the other hand, if $\alpha \in \left[n/2+1,n\right]$, then we have 
\begin{equation*}
n+1<n+2\leq 2\alpha \leq 2n <2n+1 \text{.}
\end{equation*}
Consequently, in either case, $2\alpha$ cannot be a multiple of $n+1$.
Now, 
\begin{equation*}
5n^2+7n+\left(2+4\alpha\right)=2\left(n+1\right)\left(\frac{5n}{2}+1\right)+4\alpha
\end{equation*}
so that
\begin{equation*}
\text{val}\left(f\right)=\frac{5n^2+7n+\left(2+4\alpha\right)}{2\left(n+1\right)}=\left(\frac{5n}{2}+1\right)+ \frac{2\alpha}{n+1} \text{.}
\end{equation*}
Note that $5n/2+1$ is a positive integer, since $n$ is an even integer $n \geq 6$. 
However, since $2\alpha$ is not a multiple of $n+1$, it follows from the last equation that $\text{val}\left(f\right)$ is not a positive integer, which is impossible.
\end{proof}

Observe that the preceding result excludes many graphs from being super edge-magic.
In particular, we immediately obtain the following corollary.

\begin{corollary}
\label{main_2}
For every two integers $m \geq 3$ and $n \geq 3$, the graph $C\left(m,n\right)$ is not super edge-magic when $m+n$ is odd.
\end{corollary}

At this point, a question raises naturally as the next problem indicates.

\begin{problem}
\label{Problem_1}
What can be said about the super edge-magicness of graphs of odd order $n \geq 5$ with degree sequence $s:4, 2, 2, \ldots, 2$?
\end{problem}

A partial solution to Problem \ref{Problem_1} comes directly from Lemma \ref{not_SEM}.
It is clear that a graph of odd order $n$ with degree sequence $s:4, 2, 2, \ldots, 2$ has even size. 
If the size is not only even, but also not divisible by $4$, then it meets the hypothesis of Lemma \ref{not_SEM}, 
and hence it is not super edge-magic. 
Then Problem \ref{Problem_1} becomes as follows.

\begin{problem}
\label{Problem_2}
Let $G$ be a graph of odd order $n \geq 7$ with $\left\vert E\left( G\right) \right\vert \equiv 0\pmod{4}$ and degree sequence $s:4, 2, 2, \ldots, 2$.
What can be said about the super edge-magicness of $G$?
\end{problem}

We know so far that the family of graphs $C\left(m,n\right)$ cannot be super edge-magic except possibly in the case when $m+n \equiv 0\pmod{4}$. 
This leads to propose  the following problem.

\begin{problem}
\label{Problem_3}
What can be said about the super edge-magicness of $C\left(m.n\right)$ when $m+n \equiv 0\pmod{4}$?
\end{problem}

An exhaustive computer search shows that $C\left(3,5\right)$ is not super edge-magic but is this true in general?
It is especially interesting to know whether $C\left(3,4k-3\right)$ is super edge-magic for integers $k \geq 3$.
 
\section{Conclusions}
We have established the super edge-magicness of the family of graphs $C\left(m,n\right)$ when $m+n$ is odd (see Corollary \ref{main_2}), 
and different problems have emerged out of this study. 
In particular, we want to highlight the following problem.

\begin{problem}
\label{Problem_4}
What can be said about the super edge-magicness of graphs that consist of blocks, which are all isomorphic to a cycle?
\end{problem}

We formulate the preceding problem in the most general sense possible, that is, the following concepts were introduced in \cite{LMR1}.
Let $G$ be a graph of order $p$ and size $q$ and $S_{G}$ denote the set
\begin{multline*}
  S_{G} = \Biggl\lbrace
  \frac{\sum_{u \in V\left(G\right)} \text{deg}(u)g(u) + \sum_{i=p+1}^{p+q} i }{q} \text{:}   \\
    \text{ the function } g:V\left(G\right) \rightarrow \left[1,p\right] \text{ is bijective}
  \Biggr\rbrace \text{.}
\end{multline*}
If $\lceil \min S_{G} \rceil \leq \lfloor \max S_{G} \rfloor$, then the \emph{super edge-magic interval} of $G$ is the set
\begin{equation*}
  I_G = [\lceil \min S_{G} \rceil, \lfloor \max S_{G} \rfloor] \cap \mathbb{N} \text{,}
\end{equation*}
where $\mathbb{N}$ denotes the set of natural numbers.
The \emph{super edge-magic set} of $G$ is
\begin{equation*}
  \sigma_{G} = \left\lbrace k \in I_G \right. \text{: there exists a super edge-magic labeling of } G \text{ with valence } k \left. \right\rbrace.
\end{equation*}
L\'{o}pez et al. called a graph $G$ \emph{perfect super edge-magic} if $I_G = \sigma_G$. 

Hence, we would like to encourage researchers to study which graphs, if any, in this family are perfect super edge-magic. 
It is worth to mention to finish this paper that similar ideas to the ones of perfect super edge-magic graphs had already appeared in some other works. 
The interested reader may consult the results in \cite{LMR1} and \cite{LMR2}.

\subsubsection*{$\emph{Acknowledgment}$}
The authors are gratefully indebted to Yukio Takahashi for his continuous encouragement and technical assistance during this work.

\end{document}